\documentclass[12pt]{article}

\usepackage{amssymb}
\usepackage{mathtools}
\usepackage{tikz}
\usetikzlibrary{matrix,arrows,decorations.pathmorphing,trees}
\usepackage{amsmath,color}

\usepackage[T1]{fontenc} 
\usepackage{lmodern}\usepackage{graphicx}
\usepackage{microtype} 

\usepackage[hmarginratio=1:1,top=32mm,columnsep=20pt]{geometry} 
\usepackage{hyperref}
\usepackage{amsthm}
\usepackage{url}
\usepackage{tikz}

\newtheorem{thm}{Theorem}[section]
\newtheorem{cor}[thm]{Corollary}
\newtheorem{lem}[thm]{Lemma}

\newtheorem{ex}[thm]{Example}
\newtheorem{prop}[thm]{Proposition}
\newtheorem{defn}[thm]{Definition}
\newtheorem{rem}[thm]{Remark}

\makeindex
\title{\vspace{-15mm}\fontsize{24pt}{10pt}\selectfont\textbf{Variety of Hom-Sabinin algebras and related algebra subclasses}} 

\author{
\textsc{Daniel de la Concepci\'on}\thanks{FPU Grant from the Ministerio de Educaci\'on, Cultura y Deporte in Spain.
 Research project MTM2013-45588-C3-3-P}\\[2mm] 
\normalsize Universidad de la Rioja (Spain) \\
\normalsize Departamento de Matem\'aticas y Computaci\'on\\ 
\vspace{-5mm}
\and
\textsc{Abdenacer Makhlouf}\\[2mm] 
\normalsize Universit\'e de Haute Alsace (France) \\
\normalsize Laboratoire de Math\'ematiques, Informatique et Applications\\ 
\vspace{-5mm}
}
\date{}

\begin{document}

\maketitle 

\abstract{The purpose of  this paper is to study  Sabinin algebras of Hom-type. It is  shown that Lie, Malcev, Bol and other algebras of Hom-type are naturally Sabinin algebras of Hom-type. To this end, 
we provide a general  key construction that establish a relationship  between identities of some class of Hom-algebras and ordinary algebras.  Moreover, we discuss a new concept of Hom-bialgebra, in relation with universal enveloping Hom-algebras. A study based on primitive elements is provided.}

\section*{Introduction}
Hom-type algebras appeared first in physics literature, in quantum deformations of algebras of vector fiels, mainly Witt and Virasoro algebras. It turns out that if one replaces usual derivation by $\sigma$-derivations, the Jacobi condition is no longer satisfied. The  identity satisfied is a twisted version of Jacobi condition. This was the motivation to introduce a so called Hom-Lie algebras in \cite{Silv06} and their corresponding associative algebras called Hom-associative algebras in \cite{MakhSilvHOM}. Since then many algebraic structures were extended to Hom-setting. Hom-type analogues of alternative algebras, Jordan algebras or Malcev algebras were defined and
discussed in \cite{Makh10,Yau12}. As well as $n$-ary algebras for which  Hom-type versions were introduced in   \cite{Ata09,YauLTS}.

On the other hand there has been work done in relation to the generalization of Lie algebras in such a way that these generalizations would classify differential manifolds with locally defined loop structures; as Lie algebras classify local Lie groups. This resulted in Malcev algebras in relation to Moufang loops; for instance, and eventually to Sabinin algebra and the following results from \cite{SabOr} :

\begin{itemize}
\item Analytic local loops $(Q_i,\cdot,\backslash,/,e)$ are locally isomorphic if an only if  their corresponding Sabinin $\mathfrak{q}_i=T_e(Q_i)$ algebras are isomorphic.

\item Any finite-dimensional Sabinin algebra over $\mathbb{R}$ whose structure constants satisfy certain conditions (can be found in \cite{SabOr}), is the Sabinin algebra of some analytic local loop.

\item Let $(Q,\cdot,\backslash,/,e)$ be an analytic local loop and $\mathfrak{q}$ its Sabinin algebra. If $R$ is a local subloop of $Q$ then $T_e(R)=\mathfrak{r}$ is a Sabinin sublagebra of $\mathfrak{q}$. Conversely, for any Sabinin subalgebra $\mathfrak{r}$ of $\mathfrak{q}$, there is a unique local subloop $R$ in $Q$ such that $T_e(R)=\mathfrak{r}$. The subloop $R$ is normal if and only if  $\mathfrak{r}$ is an ideal of $\mathfrak{q}$.
\end{itemize}

The structure of Sabinin algebras is extensive; as a Sabinin algebra may have an infinite number of operations. 

A Sabinin algebra is a vector space $V$ and two sets of operations $<-;-,->:V^{\otimes n}\otimes V\otimes V\to V$ for $n\geq 0$ and $\Phi_{n,m}:V^{\otimes n}\otimes V^{\otimes m}\to V$ for $n>0$ and $m>1$; that satisfy the following identities:
\begin{enumerate}
\item $<x;a,b>=-<x;b,a>$,
\item $<xaby;c,e>-<xbay;c,e>+\sum<x_{(1)}<x_{(2)};a,b>;c,e>=0$,
\item $\sigma_{a,b,c}(<xc;a,b>+<x_{(1)};<x_{(2)};a,b>,c>)=0$,
\item $\Phi(x,y)=\Phi(\tau\cdot x,\sigma\cdot y)$ where $\tau\in S_n$ and $\sigma\in S_m$,
\end{enumerate}

where $S_n$ is the permutation group of $n$ elements,  $\sum x_{(1)}\otimes x_{(2)}$ is the natural coproduct of $T(V)$, the tensor algebra, that makes it a bialgebra and where $V$ is given by   primitive elements.

The main purpose of this paper is to unify the constructions of certain classes of Hom-algebras  and study  Hom-type generalization of Sabinin algebras and related structures. A Hom-type algebra is an algebra (multiplication)  together with an endomorphism, called twisting map. The main feature of Hom-type generalizations of algebra varieties is that the endomorphism appears in the identities that define the variety; in such a way that, when
the twisting map is the identity map; one recovers the original variety of algebras.  Moreover, we discuss a new concept of Hom-bialgebra, in relation with universal enveloping Hom-algebras and provide a  study based on primitive elements.

Throughout this paper; unless specifically stated, every field $\mathbb{K}$ is of characteristic zero and an algebra is understood as a vector space $A$ with a countable set of operations $\mu_i:A^{\otimes n_i}\to A$. We mean by a Hom-algebra an algebra together with a homomorphism.

\section{Main Construction and Hom-Sabinin algebras}

In this section, we state a theorem providing a suitable way  to establish  a Hom-type generalization of a given ordinary   algebraic structure.  This will be used to study  Hom-Sabinin algebras and their subclasses. Let's define some concepts that are relevant to understand the extent of this result.

Firstly, let's recall an idea given by D. Yau in \cite{YauHomPROP} explaining how to define Hom-algebra varieties from algebra varieties using operads and PROPs. The main idea is to consider a defining identity of the variety as a function obtained using the identity function, the algebra operations; and tensor product, sum and composition of functions; then replacing a subset of instances of the identity function by endomorphisms $\alpha_i:A\to A$ to get a new identity of Hom-type. Then the Hom-algebras of the new variety are defined as those algebras with endomorphisms $(A,\{\mu_i\}_{i\in I},\{\alpha_i\}_{i\in J})$  that satisfy the new identities.

Specifically, let's recall first what a variety of algebras is:

\begin{defn}\label{Variety of algebras}\cite{SHP}
A variety of algebras is the class of all algebraic structures of a given signature satisfying a set of identities; or equivalently a class of algebraic structures of a given signature closed by homomorphic images, subalgebras and direct products. The result that states the equivalence of those two definitions is known as Birkhoff's Theorem.

A signature is a map from a fixed set $S$ to $\mathbb{N}$; $\sigma:S\to\mathbb{N}$. An algebraic structure of a signature $\sigma:S\to \mathbb{N}$ is a vector space $A$ with a set of maps identified with $S$ such that every map $s\in S$ is given as $s:A^{\otimes \sigma(s)}\to A$. In other words, the signature of an algebraic structure is a family of operations considering for each one its arity; meaning that two algebraic structures have the same signature if they have families of operations $\{f_i\}_{i\in I}$ and $\{g_i\}_{i\in J}$ of the same cardinal $I=J$ and the operation $f_i$ has the same arity as $g_i$; i.e. $\sigma(f_i)=\sigma(g_i)$.
\end{defn}

\begin{ex}
There are many examples of varieties of algebras on vector spaces:
\begin{itemize}
\item Lie algebras are given by one binary operation and two identities.
\item Associative algebras are given by one binary operation and one identity.
\item Jordan algebras are given by one binary operation and two identities.
\item $3$-Lie algebras are given by one ternary operation and three identities.
\item Bol algebras are given by one binary and one ternary operation and five identities.
\end{itemize}
\end{ex}

Varieties of algebras are interesting because of the following classical  result:

\begin{prop}
Every non-empty variety of algebras has a free algebra generated by any vector space.
\end{prop}

To make everything easier it is natural to think of free algebra elements as trees, where every node has one output and $n$ inputs. A node would then represent an $n$-ary operation of the algebraic structure.  The identities can be thought as linear combinations of those trees:

\begin{defn}\label{Hom-type} Let $(A,\{\mu_i\})$ be an algebraic structure and $\alpha:A\to A$ an endomorphism of such structure. Consider an identity as a linear combination of trees $\sum T_i$ with its internal nodes labeled by the set of operations in such a way that if the node $N_j$ is labeled with the operation $s$, then $N_j$ has $\sigma(s)$ inputs and one output. The leaves represent elements in the algebra. 
\\ The Hom-type identities would then be obtained by the following procedure:

 For every internal node $N_j$ on every tree $T_i$, apply $\alpha^{\sigma(s)-1}$ to every leaf which is not comparable to it; considering the partial order defining the tree.

Given a list of identities defining a variety of algebras, its Hom-algebra variety is defined by the Hom-type identities of those.
\end{defn}

The previous procedure can be applied to different sets of identities.

\begin{ex}
In case of associative algebras there is only one identity $(xy)z-x(yz)=0$ that is given by the following linear combination:

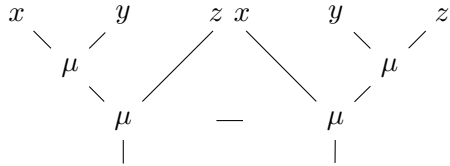
\begin{center}
\begin{tikzpicture}
\node (P0) at (-90:0.71cm) {};
\node (P1) at (0:0cm) {$\mu$} ;
\node (P2) at (45:2cm) {$z\ x$};
\node (P3) at (45+90:1cm) {$\mu$};
\node (P4) at (45+90:2cm) {$x$};
\node (P5) at (90:1.41cm) {$y$};
\node (Q0) at (-90+76:2.9cm) {};
\node (Q1) at (0:2.82cm) {$\mu$} ;
\node (Q3) at (11.31:3.61cm) {$\mu$};
\node (Q4) at (18.43:4.47cm) {$z$};
\node (Q5) at (26.57:3.16cm) {$y$};
\node (M) at (0:1.1cm) {};
\node (N) at (0:1.73cm) {};
\draw
(P0) edge[-,>=angle 90] node {} (P1)
(P1) edge[-,>=angle 90] node {} (P2)
(Q0) edge[-,>=angle 90] node {} (Q1)
(Q1) edge[-,>=angle 90] node {} (P2)
(Q1) edge[-,>=angle 90] node {} (Q3)
(P1) edge[-,>=angle 90] node {} (P3)
(Q3) edge[-,>=angle 90] node {} (Q5)
(P3) edge[-,>=angle 90] node {} (P4)
(P3) edge[-,>=angle 90] node {} (P5)
(M) edge[-,>=angle 90] node {} (N)
(Q3) edge[-,>=angle 90] node {} (Q4);
\end{tikzpicture}

\end{center}

Every internal node in this trees is binary, hence $\sigma(N_j)-1=1$. Nevertheless, only one internal node in each tree has a leaf not comparable to it. The final trees are then defined as:

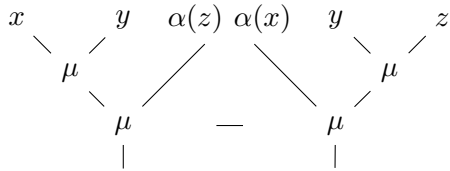
\begin{center}
\begin{tikzpicture}
\node (P0) at (-90:0.71cm) {};
\node (P1) at (0:0cm) {$\mu$} ;
\node (P2) at (45:2cm) {$\alpha(z)\ \alpha(x)$};
\node (P3) at (45+90:1cm) {$\mu$};
\node (P4) at (45+90:2cm) {$x$};
\node (P5) at (90:1.41cm) {$y$};
\node (Q0) at (-90+76:2.9cm) {};
\node (Q1) at (0:2.82cm) {$\mu$} ;
\node (Q3) at (11.31:3.61cm) {$\mu$};
\node (Q4) at (18.43:4.47cm) {$z$};
\node (Q5) at (26.57:3.16cm) {$y$};
\node (M) at (0:1.1cm) {};
\node (N) at (0:1.73cm) {};
\draw
(P0) edge[-,>=angle 90] node {} (P1)
(P1) edge[-,>=angle 90] node {} (P2)
(Q0) edge[-,>=angle 90] node {} (Q1)
(Q1) edge[-,>=angle 90] node {} (P2)
(Q1) edge[-,>=angle 90] node {} (Q3)
(P1) edge[-,>=angle 90] node {} (P3)
(Q3) edge[-,>=angle 90] node {} (Q5)
(P3) edge[-,>=angle 90] node {} (P4)
(P3) edge[-,>=angle 90] node {} (P5)
(M) edge[-,>=angle 90] node {} (N)
(Q3) edge[-,>=angle 90] node {} (Q4);
\end{tikzpicture}

\end{center}

The identity obtained is: $$(xy)\alpha(z)-\alpha(x)(yz)=0,$$ which is the identity of Hom-associative algebras.

In the sequel, we denote the Hom-type associator by 
\begin{equation}\label{Hom-associator}(x,y,z)_\alpha=(xy)\alpha(z)-\alpha(x)(yz).
\end{equation}
\end{ex}

\begin{ex}
In case of Lie algebras there are two identities $[x,y]+[y,x]=0$ and $[[x,y],z]+[[y,z],x]+[[z,x],y]=0$ that is given by the following linear combination:

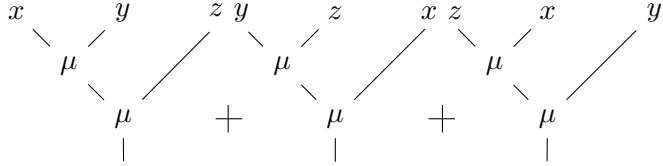
\begin{center}
\begin{tikzpicture}
\node (P0) at (-90:0.71cm) {};
\node (P1) at (0:0cm) {$\mu$} ;
\node (P2) at (45:2cm) {$x\ z$};
\node (P3) at (45+90:1cm) {$\mu$};
\node (P4) at (45+90:2cm) {$z\ y$};
\node (P5) at (90:1.41cm) {$z$};
\node (Q0) at (-90+76:2.9cm) {};
\node (Q1) at (0:2.82cm) {$\mu$} ;
\node (Q3) at (18.43:2.23cm) {$\mu$};
\node (Q4) at (18.43:4.47cm) {$y$};
\node (Q5) at (26.57:3.16cm) {$x$};
\node (O0) at (-90-76:2.9cm) {};
\node (O1) at (-180:2.82cm) {$\mu$} ;
\node (O3) at (180-11.31:3.61cm) {$\mu$};
\node (O4) at (180-18.43:4.47cm) {$x$};
\node (O5) at (180-26.57:3.16cm) {$y$};
\node (MPQ) at (0:1.1cm) {};
\node (NPQ) at (0:1.73cm) {};
\node (MOP) at (0:-1.1cm) {};
\node (NOP) at (0:-1.73cm) {};
\node (APQ) at (12.7:1.46cm) {};
\node (BPQ) at (-12.7:1.46cm) {};
\node (AOP) at (180-12.7:1.46cm) {};
\node (BOP) at (180+12.7:1.46cm) {};
\draw
(P0) edge[-,>=angle 90] node {} (P1)
(P1) edge[-,>=angle 90] node {} (P2)
(O0) edge[-,>=angle 90] node {} (O1)
(O1) edge[-,>=angle 90] node {} (P4)
(Q0) edge[-,>=angle 90] node {} (Q1)
(Q1) edge[-,>=angle 90] node {} (Q4)
(Q1) edge[-,>=angle 90] node {} (Q3)
(P1) edge[-,>=angle 90] node {} (P3)
(O1) edge[-,>=angle 90] node {} (O3)
(Q3) edge[-,>=angle 90] node {} (Q5)
(P3) edge[-,>=angle 90] node {} (P4)
(O3) edge[-,>=angle 90] node {} (O4)
(O3) edge[-,>=angle 90] node {} (O5)
(P3) edge[-,>=angle 90] node {} (P5)
(MPQ) edge[-,>=angle 90] node {} (NPQ)
(NOP) edge[-,>=angle 90] node {} (MOP)
(APQ) edge[-,>=angle 90] node {} (BPQ)
(AOP) edge[-,>=angle 90] node {} (BOP)
(Q3) edge[-,>=angle 90] node {} (P2);
\end{tikzpicture}

\end{center}

Every internal node in this trees is binary, hence $\sigma( N_j)-1=1$. In the first equation, there is no change, since the only internal node is comparable to both leaves in both trees. Nevertheless, one internal node in each tree of the second equation has a leaf not comparable to it. The final trees are then defined as:

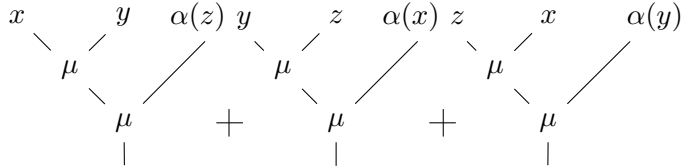
\begin{center}
\begin{tikzpicture}
\node (P0) at (-90:0.71cm) {};
\node (P1) at (0:0cm) {$\mu$} ;
\node (P2) at (45:2cm) {$\hspace*{-0.5cm}\alpha(x)\ z$};
\node (P3) at (45+90:1cm) {$\mu$};
\node (P4) at (45+90:2cm) {\hspace*{-0.5cm}$\alpha(z)\ y$};
\node (P5) at (90:1.41cm) {$z$};
\node (Q0) at (-90+76:2.9cm) {};
\node (Q1) at (0:2.82cm) {$\mu$} ;
\node (Q3) at (18.43:2.23cm) {$\mu$};
\node (Q4) at (18.43:4.47cm) {$\alpha(y)$};
\node (Q5) at (26.57:3.16cm) {$x$};
\node (O0) at (-90-76:2.9cm) {};
\node (O1) at (-180:2.82cm) {$\mu$} ;
\node (O3) at (180-11.31:3.61cm) {$\mu$};
\node (O4) at (180-18.43:4.47cm) {$x$};
\node (O5) at (180-26.57:3.16cm) {$y$};
\node (MPQ) at (0:1.1cm) {};
\node (NPQ) at (0:1.73cm) {};
\node (MOP) at (0:-1.1cm) {};
\node (NOP) at (0:-1.73cm) {};
\node (APQ) at (12.7:1.46cm) {};
\node (BPQ) at (-12.7:1.46cm) {};
\node (AOP) at (180-12.7:1.46cm) {};
\node (BOP) at (180+12.7:1.46cm) {};
\draw
(P0) edge[-,>=angle 90] node {} (P1)
(P1) edge[-,>=angle 90] node {} (P2)
(O0) edge[-,>=angle 90] node {} (O1)
(O1) edge[-,>=angle 90] node {} (P4)
(Q0) edge[-,>=angle 90] node {} (Q1)
(Q1) edge[-,>=angle 90] node {} (Q4)
(Q1) edge[-,>=angle 90] node {} (Q3)
(P1) edge[-,>=angle 90] node {} (P3)
(O1) edge[-,>=angle 90] node {} (O3)
(Q3) edge[-,>=angle 90] node {} (Q5)
(P3) edge[-,>=angle 90] node {} (P4)
(O3) edge[-,>=angle 90] node {} (O4)
(O3) edge[-,>=angle 90] node {} (O5)
(P3) edge[-,>=angle 90] node {} (P5)
(MPQ) edge[-,>=angle 90] node {} (NPQ)
(NOP) edge[-,>=angle 90] node {} (MOP)
(APQ) edge[-,>=angle 90] node {} (BPQ)
(AOP) edge[-,>=angle 90] node {} (BOP)
(Q3) edge[-,>=angle 90] node {} (P2);
\end{tikzpicture}

\end{center}

The identities obtained are: $$[x,y]+[y,x]=0$$ $$[[x,y],\alpha(z)]+[[y,z],\alpha(x)]+[[z,x],\alpha(y)]=0,$$ which define exactly  Hom-Lie algebras.

In the sequel, we denote the Hom-type Jacobiator by 
\begin{equation}
\label{Hom-Jacobiator}J_\alpha(x,y,z)=[[x,y],\alpha(z)]+[[y,z],\alpha(x)]+[[z,x],\alpha(y)].
\end{equation}
\end{ex}

We will only consider cases of homogeneous identities in our examples; i.e., identities $f(x_1,\dots,x_n)$ such that $f(\lambda_1x_1,\dots,\lambda_nx_n)=\lambda_1^{p_1}\dots\lambda_n^{p_n}f(x_1,\dots,x_n)$. It may seem restrictive, but a wide range of examples of algebras fall into this definition, in particular any Sabinin algebra. The reason is to get free algebras which are $\mathbb{N}^+$-graded.

Comparing our procedure with the one used in  \cite{YauHomPROP}; in our case we reach an algorithm which computes the Hom-type identities by changing every instance of the identity map by $\alpha^n$; where $n$ depends on the arity of the operations of the algebra.

\begin{defn}
Let $(A,\{\mu_i\})$ be an algebra. Consider $\tau$ to be a tree representing some monomial in $A$ on elements $x_1,\dots,x_n$,  then we denote the monomial obtained by the previous procedure as $(x_1\otimes\dots\otimes x_n)_\tau^\alpha$.
\end{defn}

\begin{lem}\label{lema}
Consider $\mathcal{Q}$ to be  a variety of algebras. Let $(B,\{\mu_j\},\alpha_B)$ be a multiplicative Hom-$\mathcal{Q}$-algebra which is $\mathbb{N}^+$-graded. We  define the product $\mu_{j,\alpha}:B_{a_1}\otimes\dots\otimes B_{a_{n_j}}\to B_{\sum a_i}$ as  $$\mu_{j,\alpha}(x_1,\dots,x_{n_j})=\mu_j(\alpha_B^{(\sum_{i\ne 1} a_i)-n_j+1}(x_1),\dots,\alpha_B^{(\sum_{i\ne n_j} a_i)-n_j+1}(x_{n_j})),$$ for $a_i>0$.
Under these assumptions, $(B,\{\mu_{j,\alpha}\})$ is a $\mathcal{Q}$-algebra.

If $(B,\{\mu_{j,\alpha}\})$ is a $\mathcal{Q}$-algebra for some $(B,\{\mu_j\},\alpha_B)$, $\mathbb{N}^+$-graded multiplicative Hom-algebra; then $(B,\{\mu_j\},\alpha_B)$ is in fact a Hom-$\mathcal{Q}$-algebra.
\end{lem}

\begin{proof}
The proof follows from Definition \ref{Hom-type} and the definition of $\mu_\alpha$ from the multilinear operator $\mu$.
\end{proof}

Our definition of Hom-type structures differs from some already defined Hom-algebras varieties; making $\alpha$ less generic in our definitions. The reason is that in our minds have been  universal enveloping structures and the multilinear operations. Here is an example of mismatching definitions, first we consider Lie triple systems and then 3-Lie algebras Hom-types:

\begin{ex}A Hom-Lts with this definition is given by a triple  $(L,[-,-,-],\alpha)$,  where
\begin{enumerate}
\item $[x,y,z]=-[y,x,z]$
\item $\sigma_{x,y,z}[x,y,z]=0$
\item $[\alpha^2(u),\alpha^2(v),[x,y,z]]=[[u,v,x],\alpha^2(y),\alpha^2(z)]+[\alpha^2(x),[u,v,y],\alpha^2(z)]+[\alpha^2(x),\alpha^2(y),[u,v,z]]$
\end{enumerate}

This definition differs from that in \cite{YauLTS}, as we are considering only the case where $\alpha_1=\alpha_2=\alpha^2$ in that paper. Nevertheless, if from a Hom-Lie algebra $(L,[\ ,\ ],\alpha)$ is defined a new triple product $[x,y,z]=[[x,y],\alpha(z)]$; then $(L,[\ ,\ ,\ ],\alpha)$ is a Hom-Lts in the sense defined above or $(L,[\ ,\,\ ],(\alpha^2,\alpha^2))$ is a Hom-Lts as defined in \cite{YauLTS}; meaning that our definition isn't as restrictive as it may seem.
\end{ex}

\begin{ex}A 3-Hom-Lie algebra with this definition is given by a triple  $(L,[-,-,-],\alpha)$, where
\begin{enumerate}
\item $[x,y,z]=-[y,x,z]$
\item $[x,y,z]=[y,z,x]$
\item $[\alpha^2(u),\alpha^2(v),[x,y,z]]=[[u,v,x],\alpha^2(y),\alpha^2(z)]+[\alpha^2(x),[u,v,y],\alpha^2(z)]+  [\alpha^2(x),\alpha^2(y),[u,v,z]]$
\end{enumerate}

Again this definition differs from that in \cite{Ata09}; as we are considering only the case where $\alpha_1=\alpha_2=\alpha^2$ in that paper.
\end{ex}

In both previous examples one can define the algebra variety using the following $\alpha$-twisted fundamental  identity:
\begin{align*}
& [\alpha(u),\alpha(v),[x,y,z]]= [[u,v,x],\alpha(y),\alpha(z)]+[\alpha(x),[u,v,y],\alpha(z)]+\\ & 
\ \ \ \ \ \ \ \ \ \ \ \ \ \ \ \ \ \ \ \  \ \ \ \ \ \ \ \  [\alpha(x),\alpha(y),[u,v,z]],
\end{align*}
but the results obtained in the following sections cannot be applied to this more general case; as there might not be an endomorphism $\beta$ such that $\alpha=\beta^2$.\\

The definition of Hom-Sabinin algebras results in:

\begin{defn}\label{defSab} A Sabinin algebra of Hom-type, or Hom-Sabinin algebra $(S,\{<-;-,->\}_n,\{\Phi_{n,m}\}_{n,m},\alpha)$ is defined as:

\begin{itemize}
\item a vector space $S$,
\item a family of maps $<-;-,->:S^{\otimes n}\otimes S\otimes S\to S$ for $n\geq 0$ and $\Phi_{n,m}:S^{\otimes n}\otimes S^{\otimes m}\to S$ for $n>0$ and $m>1$,
\item an endomorphism $\alpha$ of the algebra $(S,\{<-;-,->:S^{\otimes n}\otimes S\otimes S\to S\}_n,\{\Phi_{n,m}:S^{\otimes n}\otimes S^{\otimes m}\to S\}_{n,m})$,
\end{itemize}

such that the following identities hold:
\begin{eqnarray}
& \label{Hsab1}<x;a,b>+<x;b,a>=0,\\
&\label{Hsab2} <x[a,b]y;c,e>+\sum<\alpha^{k}(x_{(1)})<x_{(2)};a,b>\alpha^{k}(y);\alpha^{k}(c),\alpha^{k}(e)>=0, \\
&\label{Hsab3} \sigma_{a,b,c}(<xc;a,b>+\sum<\alpha^{k}(x_{(1)});<x_{(2)};a,b>,\alpha^{k}(c)>)=0,\\
&\label{Hsab4}  \Phi(x,y)=\Phi(\tau\cdot x,\sigma\cdot y),
\end{eqnarray} where $\tau\in S_n$,  $\sigma\in S_m$
and  $k=|x_{(2)}|+1$.
\end{defn}

The following result provides
the main theorem of this paper; not because of its elaborate proof, which is rather easy, but because it has very deep consequences.

\begin{thm}[Main Theorem]
Consider some variety of algebras, $\mathcal{Q}$, where every free Hom-$\mathcal{Q}$-algebra is $\mathbb{N}^+$-graded.
Then any Hom-$\mathcal{Q}$-algebra satisfies the Hom-type identities satisfied by the $\mathcal{Q}$-algebras.
\end{thm}

\begin{proof}
By our first assumption, we only need to prove the result for free algebras.

Consider any free Hom-$\mathcal{Q}$-algebra $B(V,\alpha)$. This Hom-algebra is the quotient of the free Hom-algebra $F(V,\alpha)$ by the ideal generated by the identities that define Hom-$\mathcal{Q}$-algebras. By our second assumption, it is graded by $\mathbb{N}^+$. It is then possible to define $\{\mu_{j,\alpha}\}$ as in the previous results.

By Lemma \ref{lema}, $(B(V),\mu_{j,\alpha})$ is a $\mathcal{Q}$-algebra; and hence satisfies all identities of $\mathcal{Q}$-algebras. The result follows from unfolding the definition of $\mu_{j,\alpha}$ in any identity of $\mathcal{Q}$-algebras satisfied by $(B(V),\mu_{j,\alpha})$.
\end{proof}

Note that, as stated before, if the defining relations of the variety $\mathcal{Q}$ are given by homogeneous multilinear identities, the free $\mathcal{Q}$-algebras are  $\mathbb{N}^+$-graded and so are the free Hom-$\mathcal{Q}$-algebras. In conclusion, this result can be directly applied to those algebra classes.

One may provide a generalization of the Yau's twisting principle, given in \cite{DYHomology,YauHomPROP}.

\begin{prop}\label{endoCons}
Let $(S,\alpha)$ be a Hom-Sabinin algebra and  $\beta:S\to S$ be an endomorphism of the Hom-Sabinin algebra. \\
 Then $(S,\{\beta^{n+1}\circ<-;-,->\}_n,\{\beta^{n+m-1}\circ \Phi_{n,m}\}_{n,m},\beta\circ\alpha)$ is a Hom-Sabinin algebra.
\end{prop}








In particular, one may construct a Hom-Sabinin algebra starting from a Sabinin algebra and a Sabinin algebra endomorphism.
\begin{cor}
Let $(S,\{<-;-,->:S^{\otimes n}\otimes S\otimes S\to S\}_n)$ be a Sabinin algebra and $\alpha:S\to S$ be a Sabinin algebra endomorphism.\\  Then $(S,\{\alpha^{n+1}\circ<-;-,->:S^{\otimes n}\otimes S\otimes S\to S\}_n,\{\alpha^{n+m-1}\circ \Phi_{n,m}\}_{n,m} ,\alpha)$ is a Hom-Sabinin algebra.
\end{cor}


\section{Consequences of the Main Theorem}

\subsection{Hom-Sabinin algebra subclasses}

Lie algebras, Malcev algebras and  Bol algebras are particular examples of Sabinin algebras. The natural question then is whether the Hom-type algebras of Lie, Malcev and Bol are Hom-Sabinin and the answer is affirmative.

\begin{ex}
Let $(L,[\ ,\ ],\alpha)$ be a multiplicative  Hom-Lie algebra. Define
\begin{eqnarray*}
&& <1;a,b>=<a,b>=-[a,b], \\&&  <x;a,b>=0,\\ && \Phi(x,y)=0,
\end{eqnarray*}
 for $a,b\in L$ and $x,y\in L^{\otimes n}$, where $n>0$. \\ Then $L$ becomes a Hom-Sabinin algebra.
\end{ex}

\begin{ex} 

A Hom-Malcev algebra, see  \cite{Elhamda-Mak,Yau12},  is a Hom-algebra $(M,[\ , \ ],\alpha )$ where $ [\ ,\  ]:M\times M\rightarrow M$ is a skewsymmetric bilinear map and $\alpha : M \rightarrow  M$  an algebra map with the respect to the bracket  satisfying for all $x, y, z \in M$
\begin{equation}\label{MalcevIdentity}
J_\alpha(\alpha(x), \alpha(y), [x, z]) = [J_\alpha(x, y, z), \alpha^2 (x)] ,
\end{equation} where $J_\alpha$ is the Hom-type Jacobiator.

Let $(M,[\ ,\ ],\alpha)$ be a Hom-Malcev algebra.  Define 
\begin{eqnarray*}
 &&<1;a,b>=<a,b>=-[a,b] ,\\ &&<c;a,b>=-\dfrac{1}{3}J_\alpha(a,b,c),\\ &&<xc;a,b>=\sum<\alpha^{|x_{(2)}|+1}(x_{(1)});\alpha^{|x_{(2)}|+1}(c),<x_{(2)};a,b>>,\\&&\Phi(x,y)=0,
\end{eqnarray*}
 for $a,b,c\in M$ and $x,y\in M^{\otimes n}$, where $n>0$. \\ Then $M$ becomes a Hom-Sabinin algebra.
\end{ex}

\begin{ex} 
A Hom-Bol algebra (see \cite{Attan-Issa}) is a quadruple  $(M, [\ ,  \ ],\{ ,\ ,\  , \} , \alpha )$ where $ [\, , \  ]:M\times M\rightarrow M$  is a  skewsymmetric bilinear map,   $ \{ \  , \  , \ \} :M\times M\times M \rightarrow M$    is a trilinear  map and $\alpha : M \rightarrow  M$  an algebra map with the respect to the brackets  satisfying for all $x, y, z, u, v,w  \in M$
\begin{eqnarray}\label{BolIdentity1}
& \{ x, y, z \}=-\{ y,x, z \}, \\
& \{x, y, z\}+\{z, x, y\}+\{y, z, x\}=0,\\
& \{\alpha (x), \alpha (y), [u , v] \} = [ \{x, y, u\} , \alpha^2(v) ] +[ \alpha^2 (u) , \{x, y, v\}]  + \{ \alpha (u),\alpha(v), [x , y]\}\nonumber \\ & - [[\alpha(u),\alpha(v) ],[ \alpha(x),\alpha(y)]],\\
& \{\alpha^2(x), \alpha^2(y), \{u, v, w\}\} = \{\{x, y, u\}, \alpha^2(v), \alpha^2(w)\} + \{\alpha^2(u), \{x, y, v\}, \alpha^2(w)\} \nonumber \\ &+\{\alpha^2(u), \alpha^2(v), \{x, y, w\}\} .
\end{eqnarray}

Let $(B,[\ ,\ ],[\ ,\ ,\ ],\alpha)$ be a left Hom-Bol algebra. Define  
\begin{eqnarray*}
&& <1;a,b>=<a,b>=-[a,b],\\&&<c;a,b>=\{a,b,c\}-[[a,b],\alpha(c)],\\&&<cx;a,b>=-\sum<\alpha^{|x_{(2)}|+1}(x_{(1)});\alpha^{|x_{(2)}|+1}(c),<x_{(2)};a,b>>,\\&&\Phi(x,y)=0,
\end{eqnarray*} for $a,b,c\in B$ and $x,y\in B^{\otimes n}$, where $n>0$.\\ Then $B$ becomes a Hom-Sabinin algebra.
\end{ex}

\begin{ex}
A Hom-Lie Yamaguti algebra (see \cite{Gaparayi-Issa}) is a quadruple  $(A, [\ ,  \ ],\{ ,\ , , \} , \alpha )$ where $ [\, , \  ]:A\times A\rightarrow A$  is a  skewsymmetric bilinear map,   $ \{ \ ,\ , \ \} :A\times A\times A \rightarrow A$ is a trilinear  map and $\alpha : A \rightarrow  A$  an algebra map with the respect to the brackets   satisfying  for all $x, y, z, u, v  \in A$
\begin{eqnarray}\label{BolIdentity1}
& \{ x, y, z \}=-\{ y,x, z \}, \\
&  [[x,y],\alpha(z)]+ [[z,x],\alpha(y)]+ [[y,z],\alpha(x)]\nonumber \\ & +\{x, y, z\}+\{z, x, y\}+\{y, z, x\}=0,\\
& \{ [x, y],\alpha(z),\alpha(u)\}+ \{ [z, x],\alpha(y),\alpha(u)\}+ \{ [y, z],\alpha(x),\alpha(u)\}=0,\\
& \{\alpha(x),\alpha(y),[u,v]\} = [\{x,y,u\},\alpha^2(v)] + [\alpha^2(u),\{x,y,v\}], \\
& \{ \alpha^2(u),\alpha^2(v),\{x,y,z\}\}=\{\{u,v,x\},\alpha^2 (y),\alpha^2(z)\}+ \{\alpha^2(x),\{u,v,y\},\alpha^2(z)\}\nonumber\\ 
& +\{\alpha^2(x),\alpha^2 (y),\{u,v,z\}\}.
\end{eqnarray}

Let $(A,[\ ,\ ],\{\ ,\ ,\ \},\alpha)$ be a Hom-Lie-Yamaguti algebra. Define 
\begin{eqnarray*}
 && <1;a,b>=<a,b>=-[a,b],\\&&<c;a,b>=\{a,b,c\},\\&&<xc;a,b>=<\alpha(x);\alpha(c),[a,b]>\\ && \quad\quad  +\sum<\alpha^{|x_{(2)}|+1}(x_{(1)});\alpha^{|x_{(2)}|+1}(c),<x_{(2)};a,b>>,\\&&\Phi(x,y)=0,
 \end{eqnarray*} for $a,b,c\in A$ and $x,y\in A^{\otimes n}$, where $n>0$. \\Then $A$ becomes a Hom-Sabinin algebra.
\end{ex}

\subsection{Hom-Akivis algebras and Hom-BTQQ algebras}

There are examples of more algebras where their Hom-type structure has been defined and that are related to Sabinin algebras, one such examples is Akivis algebras; since every Sabinin algebra is an Akivis algebra. There are also examples of algebras related to Sabinin algebras whose Hom-type structure have not yet been defined; for instance, BTQQ algebras. Their name comes from the fact that they have a Binary, a Ternary and two Quaternary products. This algebras appeared when considering algebras that are Sabinin up to some degree $d$; i.e., they satisfy the Sabinin defining equations that involve up to $d$-linear operations and forget about the rest. This type of algebras were worked on up to degree $4$ in \cite{SaraBTQQ}, using computational methods.

In this setting, Sabinin algebras of degree $2$ are skew-symmetric algebras, Sabinin algebras of degree $3$ are Akivis algebras and Sabinin algebras of degree $4$ are BTQQ algebras. 

\begin{defn}
An Akivis algebra is a vector space $V$ endowed with two linear maps $[\ ,\ ]:V^{\otimes 2}\to V$ and $(\ ,\ ,\ ):V^{\otimes 3}\to V$ such that the following identities hold:
\begin{eqnarray}
& [a,b]+[b,a]=0,\\ &
[[a,b],c]+[[b,c],a]+[[c,a],b]=   (a,b,c)-(a,c,b)-(b,a,c)+(b,c,a)\nonumber \\ & +(c,a,b)-(c,b,a).
\end{eqnarray}
\end{defn}

\begin{defn}
A BTQQ algebra is a vector space $V$ endowed with four linear maps $[\ ,\ ]:V^{\otimes 2}\to V$, $(\ ,\ ,\ ):V^{\otimes 3}\to V$, $\{\ ,\ ,\ ,\ \}:V^{\otimes 4}\to V$ and $[|\ ,\ ,\ ,\ |]:V^{\otimes 4}\to V$ such that $(V,[\ ,\ ],(\ ,\ ,\ ))$ is an Akivis algebra and the following identities hold:
\begin{eqnarray}
& ([a,b],c,d)-[a,(b,c,d)]+[b,(a,c,d)]=\{a,b,c,d\}-\{b,a,c,d\},\\
& (a,[b,c],d)-[b,(a,c,d)]+[c,(a,b,d)]=[|a,b,c,d|]-[|a,c,b,d|],\\ 
& [b,(a,c,d)]-[b,(a,d,c)]-(a,b,[c,d])=   \{a,b,c,d\}-\{a,b,d,c\} \nonumber\\ & -[|a,b,c,d|]+[|a,b,d,c|].
\end{eqnarray}
\end{defn}

\begin{lem}\cite{SaraBTQQ}
BTQQ algebras are equivalent to Sabinin algebras of degree $4$.
\end{lem}

If we consider  Hom-Sabinin algebras of degree $3$, one gets Hom-Akivis algebra, as  mentioned in \cite{IssaAkivis}. Consider the identity   \eqref{Hsab3} on the definition of a Hom-Sabinin algebra, applied to three arbitrary elements $a,b,c$ and $x=1$, one gets  the Hom-Akivis identity. Hom-Sabinin algebras of degree $4$  leads to  Hom-BTQQ algebras that are defined as:

\begin{defn}
A Hom-BTQQ algebra is a vector space $V$ endowed with four linear maps $[\ ,\ ]:V^{\otimes 2}\to V$, $(\ ,\ ,\ ):V^{\otimes 3}\to V$, $\{\ ,\ ,\ ,\ \}:V^{\otimes 4}\to V$ and $[|\ ,\ ,\ ,\ |]:V^{\otimes 4}\to V$; and an endomorphism $\alpha:V\to V$s such that the following equations hold:
\begin{eqnarray}
& [a,b]+[b,a]=0,\\
& [[a,b],\alpha (c)]+[[b,c],\alpha (a)]+[[c,a],\alpha (b)]=(a,b,c)-(a,c,b)-(b,a,c)\nonumber \\ & +(b,c,a)+(c,a,b)-(c,b,a),\\
& ([a,b],\alpha (c),\alpha (d))-[\alpha^2 (a),(b,c,d)]+[\alpha^2 (b),(a,c,d)]= \nonumber \\&\{a,b,c,d\}-\{b,a,c,d\},\\
& (\alpha (a),[b,c],\alpha (d))-[\alpha^2 (b),(a,c,d)]+[\alpha^2 (c),(a,b,d)]=\nonumber \\ & [|a,b,c,d|]-[|a,c,b,d|]\\ 
& [\alpha^2 (b),(a,c,d)]-[\alpha^2 (b),(a,d,c)]-(\alpha (a),\alpha (b),[c,d])=\nonumber \\ & \{a,b,c,d\}-\{a,b,d,c\}-[|a,b,c,d|]+[|a,b,d,c|].
\end{eqnarray}
\end{defn}

\begin{thm}
Hom-BTQQ algebras are equivalent to Hom-Sabinin algebras of degree $4$.
\end{thm}
\begin{proof}
The proof given in \cite{SaraBTQQ} for the non Hom-type case can be copied for the Hom-type case; using the Main Theorem.
\end{proof}

\begin{cor}
Let $(V,[\ ,\ ],(\ ,\ ,\ ),\{\ ,\ ,\ ,\ \},[|\ ,\ ,\ ,\ |],\alpha)$ be a Hom-BTQQ algebra and $\beta:V\to V$ be an endomophism of the Hom-algebra $V$. Then $(V,\beta[\ ,\ ],\beta^2(\ ,\ ,\ ),\beta^3\{\ ,\ ,\ ,\ \},\beta^3[|\ ,\ ,\ ,\ |],\beta\circ\alpha)$ is a Hom-BTQQ algebra.
\end{cor}

\begin{proof}
Is a consequence of the same result for Hom-Sabinin algebras and the fact that Hom-BTQQ algebras are Hom-Sabinin algebras of degree $4$; using the Main Theorem.
\end{proof}

This procedure does not always give nice examples as it may seem. We present in the following an example where it leads to a trivial Hom-Akivis algebra.

\begin{ex}
Consider the following Hom-algebra $V$ over $\mathbb{C}$ generated by $\langle x,y\rangle $ with a trilinear operation $$(x,x,x)=0=(x,y,x)=(y,x,y)=(y,y,y),$$ $$(x,x,y)=-2x-4y=-(y,x,x)=2(y,y,x)=-2(x,y,y).$$

Then this algebra is an Akivis algebra by considering $[a,b]=0$.

Set $\alpha(x)=-2\alpha(y)=-2x$, which  provides an Akivis algebra morphism; leading to a Hom-Akivis algebra defined by  $$(x,x,x)=0=(x,y,x)=(y,x,y)=(y,y,y),$$ $$(x,x,y)=-2\alpha^2(x)-4\alpha^2(y)=-(y,x,x)=2(y,y,x)=-2(x,y,y).$$

Since $-2\alpha(x)-4\alpha(y)=0$, then the resulting Hom-Akivis algebra is trivial.
\end{ex}

\begin{ex}
Consider now any Lie algebra $(L,\mu)$ and a Lie algebra morphism $\alpha$. Then the operations $[\ ,\ ]_\mu$, the commutator of $\mu$, and $(\ ,\ ,\ )_\alpha$, the Hom-associator of $\mu$, form a Hom-Akivis algebra.\\ In particular,  the operations are:
\begin{eqnarray*}
& [a,b]_\mu=2\mu(a,b),\\
& (a,b,c)_\alpha=\mu(\alpha(b),\mu(c,a)).
\end{eqnarray*}

A more specific example would be $\mathfrak{sl}_2(\mathbb{C})$ as the Lie algebra and $\alpha(h)=-h$, $\alpha(x)=y$ and $\alpha(y)=x$. The formulas for the non zero new operations in this particular case are:

$$[x,y]=2h,\ [h,x]=4x,\ [h,y]=-4y$$
$$(x,x,y)=-(y,x,x)=-2y,\ (x,x,h)=-(h,x,x)=-2h,$$ $$(x,y,y)=-(y,y,x)=2x,\ (h,y,y)=-(y,y,h)=2h,$$ $$(h,h,x)=-(x,h,h)=4x,\ (h,h,y)=-(y,h,h)=4y.$$
\end{ex}

The same constructions can be consider to form BTQQ Hom-algebras from any Hom-algebra using the Hom-type formulas of those given in \cite{SaraBTQQ}.

\subsection{Hom-power associative algebras and Hom-Type Teichm\"uller identity}

There are several consequences of the Main Theorem already proven in the literature, such as the example of Hom-power associative algebras.

\begin{thm}\cite[Corollary 5.2]{YauPowAss}
Let $(A, \mu, \alpha)$ be a multiplicative Hom-algebra. Then the following statements are equivalent.

\begin{enumerate}
\item $A$ satisfies
$x^2\alpha(x) = \alpha(x)x^2$ and $x^4 = \alpha(x^2)\alpha(x^2)$
for all $x\in A$.
\item $A$ satisfies
$(x, x, x)_\alpha = 0 = (x^2, \alpha(x), \alpha(x))_\alpha$
for all $x\in A$.
\item $A$ is up to fourth Hom-power associative.
\item $A$ is Hom-power associative.
\end{enumerate}
\end{thm}

\begin{proof}
Firstly, the equations on $1$ and $2$ are the same equations writen in two different ways, hence $1\Leftrightarrow 2$. It is trivial that $4\Rightarrow 3\Rightarrow 2$.

The fact $2\Rightarrow 4$ is true for power associative algebras as proven in \cite{AlbertPAss}. Also, the free power associative algebras are $\mathbb{N}^+$-graded since their defining relations are homogeneous polynomials. Hence by application of the Main Theorem, it follows that the result is true since they are the Hom-type identities of the power associative case.
\end{proof}

In the following, we recover the Hom-Type Teichm\"uller identity  \cite{YauRAlt}.

\begin{thm}\cite[Lemma 7.3]{YauRAlt}
Let $(A, \mu, \alpha)$ be a multiplicative Hom-algebra. Then for any $w,x,y,z\in A$ the following equation holds:

$$(wx,\alpha (y),\alpha (z))_\alpha-(\alpha (w),xy,\alpha (z))_\alpha+(\alpha( w),\alpha (x),yz)_\alpha$$ $$-\alpha^2(w)(x,y,z)_\alpha-(w,x,y)_\alpha\alpha^2(z)=0.$$
\end{thm}

\begin{proof}
The Teichm\"uller identity holds for any algebra. The previous equation is its Hom-type identity, hence it holds by application of the Main Theorem.
\end{proof}

\section{Towards  universal enveloping algebras of Hom-Sabinin algebras}

We aim to discuss universal enveloping algebras of  Hom-Sabinin algebras. We consider the  functor $U:SabAlg\to BiAlg_{cc}$; from Sabinin algebras to cocommutative, coassociative and magmatic bialgebras. We show that it is adjoint to the functor $YIII$, defined in \cite{YIII}, that gives a Sabinin algebra structure to any algebra. This construction is in fact an equivalence between Sabinin algebras and connected cocommutative bialgebras. In this section, we will construct the functor $YIII_{hom}$.

\begin{defn}\label{YIII}
Consider the operations $q^\alpha_{n,m}$ that make the equation 
\begin{align*} & (\alpha^{v-1}[u]_\alpha,\alpha^{u-1}[v]_\alpha,\alpha^{u+v-2}z)_\alpha=q^\alpha(u,v,z)\\ &+\sum_{u_1+v_1>0} \alpha^{u_2+v_2}(\alpha^{v_1-1}([u_{(1)}]_\alpha)\alpha^{u_1-1}([v_{(1)}]_\alpha))\alpha^{u_1+v_1-1}(q^\alpha(u_{(2)},v_{(2)},z)),
\end{align*} hold.

Then we define $YIII_{hom}(A,\mu,\alpha)$ as 
\begin{eqnarray*}
& <u;a,b>=-q^\alpha(u,a,b)+q^\alpha(u,b,a),\\ & <a,b>=-[a,b], \text{ and also} \\ & \Phi_{n,m}(u,v)=\sum_{\sigma\in S_n,\tau\in S_m}\frac{1}{n!m!}q^\alpha_{n,m-1}(u_{\sigma(1)} \dots u_{\sigma(n)},v_{\tau(1)}\dots v_{\tau(m-1)},v_{\tau(m)}).
\end{eqnarray*}

In this definition, we consider that $[x_1\otimes\dots \otimes x_n]_\alpha=(x_1\otimes\dots \otimes x_n)^\alpha_\tau$ where $\tau$ is the tree of right normed words: $(x_1\otimes \dots\otimes x_n)_\tau=(\dots((x_1x_2)x_3)\dots )x_n$.
\end{defn}

Note that in case $\alpha=id$, we recover the original functor from \cite{YIII}.

\begin{ex}
In case $(A,\mu,\alpha)$ is Hom-associative, then $YIII_{hom}(A,\mu,\alpha)$ is the corresponding  Hom-Lie algebra $A^-$.
\end{ex}

\begin{ex}
In case $(A,\mu,\alpha)$ is Hom-alternative, then $YIII_{hom}(A,\mu,\alpha)$ is the corresponding  Hom-Malcev algebra $A^-$; as $$-q(c,a,b)+q(c,b,a)=2(c,b,a)_\alpha=-J_\alpha(a,b,c)-4(c,b,a)_\alpha .$$ Concluding that $<c;a,b>=-\dfrac{1}{3}J_\alpha(a,b,c)$.
\end{ex}

\begin{ex} Here are the first computations of the $q^\alpha$ operations:
\begin{itemize}
\item In case $n=1$, $m=1,$ one gets $q_{1,1}^\alpha(x;y;z)=(x,y,z)_\alpha.$
\item In case $n=2$, $m=1,$ one gets $$q_{2,1}^\alpha(x,y;t;z)=(xy,\alpha(t),\alpha(z))_\alpha-\alpha^2(x)(y,t,z)_\alpha-\alpha^2(y)(x,t,z)_\alpha.$$
\item In case $n=1$, $m=2$ one gets $$q_{1,2}^\alpha(x;y,t;z)=(\alpha(x),yt,\alpha(z))_\alpha-\alpha^2(y)(x,t,z)_\alpha-\alpha^2(t)(x,y,z)_\alpha.$$
\end{itemize}
\end{ex}

\subsection{Hom-bialgebras and their primitive elements}

In this section, we define and study (nonassociative) Hom-bialgebras and some properties of their primitive elements.

\begin{defn}
A coalgebra is a pair $(A,\Delta)$ where $A$ is a $\mathbb{K}$-vector space and $\Delta:A\to A\otimes A$ is a linear map.
\end{defn}

\begin{defn}
Let  $(A,\mu,\alpha)$ be a  Hom-algebra, with a binary operation. We say that $A$ is unitary if there is a linear map $u:\mathbb{K}\to A$ such that $$\mu(u(1),x)=\mu(x,u(1))=\alpha(x)$$ for any $x\in A$, and $\alpha\circ u=u$.
\end{defn}

\begin{defn}
A Hom-coalgebra is a tuple $(A,\Delta,\alpha)$ where $(A,\Delta)$ is a coalgebra and  $\alpha:A\to A$ is a coalgebra morphism. We say that a Hom-coalgebra is counitary if there is a map $\epsilon:A\to\mathbb{K}$ such that $$\sum \epsilon(x_{(1)})\otimes x_{(2)})=1\otimes x \text{ and } \sum x_{(1)}\otimes \epsilon(x_{(2)})=x\otimes 1.$$
\end{defn}

\begin{defn}
A Hom-bialgebra is a sextuple  $(A,\mu,u,\Delta,\epsilon,\alpha)$, where $(A,\mu,u,\alpha)$ is a unitary Hom-algebra, $(A,\Delta,\epsilon,\alpha)$ is a counitary Hom-coalgebra and the maps $\epsilon:A\to \mathbb{K}$ and $\Delta:A\to A\otimes A$ are Hom-algebra morphisms.\\
 As a consequence, the relation $\sum x_{(1)} u(\epsilon(x_{(2)}))=\sum u(\epsilon(x_{(1)})) x_{(2)}=\alpha(x)$ follows; equivalently $\mu\circ(id\otimes (u\circ\epsilon))\circ\Delta=\mu\circ((u\circ\epsilon)\otimes id)\circ\Delta=\alpha$.

Note that in case $\alpha$ is the identity map, one recovers the definition of a unitary and counitary bialgebra.
\end{defn}
\begin{rem}
Notice that this definition is different from the usual one where the multiplication is assumed to be Hom-associative and the comultiplication is Hom-coassociative \cite{MakhSilvHOMBialg}. 
\end{rem}

\begin{lem}\label{BFree}
Let $B$ be a free unitary Hom-algebra generated by a set  $G$. 

Define $\Delta_\alpha(p)=u(1)\otimes p+p\otimes u(1)$ for any $p\in G$ and extend it to $B$ as an algebra morphism. In particular, $\Delta_\alpha(u(1))=u(1)\otimes u(1)$. Define also $\epsilon(p)=0$ for $p\in G$ and $\epsilon(u(1))=1$. 

Then $\mathcal{B}=(B,\mu,u,\Delta_\alpha,\epsilon,\alpha)$ is a coassociative, cocommutative, counitary and unitary Hom-bialgebra.
\end{lem}
\begin{proof}
The proof is straightforward.
\end{proof}

Note that the difference between the Hom-bialgebra structure of a free Hom-algebra and the bialgebra structure of a free algebra is mainly the definition of the adjoint map of the unit element, that is a group-like element of the coalgebra structure.

Since all  bialgebras we are going to consider are cocommutative, we shall rewrite $a\otimes b+b\otimes a$ as $a\circ b$ so that our equations are half the size. Note that  $a\circ b=b\circ a$.

\begin{lem}
If $p$ is a primitive element of a  Hom-bialgebra $(A,\mu,u,\Delta,\epsilon,\alpha)$, then $\alpha(p)$ is also primitive.
\end{lem}

\begin{proof}
$\Delta_\alpha(\alpha(p))=\alpha^{\otimes 2}(\Delta_\alpha(p))=u(1)\otimes \alpha(p)+\alpha(p)\otimes u(1)$.
\end{proof}

\begin{thm}
Let $p$ be a primitive element of $(B,\mu,1,\Delta,\varepsilon)$, the free algebra generated by a set $G$ with its natural bialgebra structure. Then $p^\alpha$, the Hom-type element obtained by our procedure in \ref{Hom-type} applied to $p$, is a primitive element of $(B,\mu,u,\Delta_\alpha,\epsilon,\alpha)$, the free Hom-algebra generated by $(\mathbb{K}\langle G\rangle,\alpha)$ with its natural Hom-bialgebra structure.
\end{thm}

\begin{proof}
Consider $p=\sum\lambda_im_i$ some element in a free algebra $B$ where $m_i$ are nonassociative words on the generators and $\lambda_i\in\mathbb{K}$. In other words, $p$ is a non-associative polynomial. If $p$ is primitive, then $\sum\lambda_i \Delta_\alpha(m_i)=\sum\lambda_i( 1\otimes m_i+m_i\otimes 1)$. Let $m_i=(x_{i,1}\otimes \dots\otimes x_{i,n_i})_{\tau_i}$;
then $\Delta_\alpha(m_i)=(\Delta_\alpha(x_{i,1}\otimes \dots\otimes x_{i,n_i}))_{\tau_i}$.

In conclusion, one gets the following equation on $B\otimes B$: 
\begin{align*} & \sum\lambda_i((1\circ x_{i,1})\otimes \dots\otimes(1\circ x_{i,n_i}))_{\tau_i}\\ &=\sum\lambda_i(((1\otimes x_{i,1})\otimes \dots\otimes(1\otimes x_{i,n_i}))_{\tau_i}+((x_{i,1}\otimes1)\otimes \dots\otimes (x_{i,n_i}\otimes 1))_{\tau_i}),
\end{align*}
where $a\circ b=a\otimes b+b\otimes a$.

By the Main Theorem; since $F\otimes F$ is free, the Hom-type equation obtained by applying our procedure to the one above is also true for a free Hom-algebra. Indeed,
\begin{align*}
& \sum\lambda_i((1\circ x_{i,1})\otimes \dots\otimes(1\circ x_{i,n_i}))^\alpha_{\tau_i}\\ & =\sum\lambda_i(((1\otimes x_{i,1})\otimes \dots\otimes(1\otimes x_{i,n_i}))^\alpha_{\tau_i}+((x_{i,1}\otimes1)\otimes \dots\otimes (x_{i,n_i}\otimes 1))^\alpha_{\tau_i}).
\end{align*}

Undoing the coproducts, the equation translates to:
\begin{align*} & \sum\lambda_i\Delta_\alpha((x_{i,1}\otimes \dots\otimes {}_ix_{n_i})^\alpha_{\tau_i})\\ &=\sum\lambda_i(u(1)\otimes (x_{i,1}\otimes \dots\otimes x_{i,n_i})^\alpha_{\tau_i})+((x_{i,1}\otimes \dots\otimes x_{i,n_i})^\alpha_{\tau_i}\otimes u(1)).
\end{align*}

Therefore, the Hom-type monomial of $p=\sum\lambda_im_i$ is primitive for the Hom-bialgebra.
\end{proof}

\begin{cor}
The operations $q^\alpha:B^{\otimes n}\otimes B^{\otimes m}\otimes B \to B$ defined in \ref{YIII} are primitive in the free Hom-algebra $B$ generated by $G=\{x_1,\dots, x_n,\dots\}\cup\{y_1,\dots, y_n,\dots\}\cup\{z\}$ and an endomorphism $\alpha:\mathbb{K}\langle G\rangle\to \mathbb{K}\langle G\rangle$.
\end{cor}

\begin{cor}
The operations $q^\alpha:C^{\otimes n}\otimes C^{\otimes m}\otimes C\to C$ are primitive in the Hom-bialgebra $(C,\mu,u,\delta,\epsilon,\alpha)$.
\end{cor}

\begin{proof}
Let $a_1,\dots,a_n,b_1,\dots,b_m,c$ be primitive elements for $(C,\delta)$ and define $\phi(x_i)=a_i$, $\phi(y_i)=b_i$, $\phi(z)=c$ and $\phi(u)=0$ for any other $u\in G$. Also take any $\beta:\mathbb{K}\langle G\rangle \to \mathbb{K}\langle G\rangle$ such that $\phi\circ\beta=\alpha\circ \phi$.

Then extend $\phi:B\to C$ as a Hom-algebra morphism. Since $\phi(G)$ are primitive elements of $(C,\delta)$, it follows that $(\phi\otimes \phi)\circ\Delta=\delta\circ \phi$ on $G$; but then it is true for every element since $G$ generates $B$ as an algebra and they are algebra morphisms.

Finally, let's compute $\delta(q^\alpha(a,b,c))$:
\begin{eqnarray*}
&&\delta(q^\alpha(a,b,c))=\delta(q^\alpha(\phi(x),\phi(y),\phi(z)))=\delta(\phi(q^\beta(x,y,z)))\\ &&=((\phi\otimes \phi)\circ\Delta)(q^\beta(x,y,z))=1\circ \phi((q^\beta(x,y,z)))=1\circ (q^\alpha(a,b,c)).
\end{eqnarray*}
\end{proof}

The combinatorics in relation with $\alpha$ in the coproduct of non-associative words are very interesting and has deep consequences in the structure of Hom-bialgebras. Here are some properties:

Consider a right-normed word $u=((\dots(x_1x_2)\dots x_{n-1})x_n)$. The question is: How do $\alpha$ appear in a particular summand of $\Delta(u)$?

First, consider only three generators $(xy)z$, it follows then that $\Delta((xy)z)=(xy)z\circ 1+\alpha^{\otimes 2}(xy\circ z)+\alpha(y)z\circ \alpha^2(x)+\alpha(x)z\circ \alpha^2(y)$. We realize that $\alpha$ keeps some record of the order of multiplications of the elements.

In every summand, the left and right elements of the tensor are right-normed words and the instances of $\alpha$ appearing in each side of the tensor indicate where there are elements missing from the original word.

Here is an algorithm to compute this elements for every type of non-associative monomial using trees:

\begin{lem}
Let $m=(x_1\otimes\dots\otimes x_n)_\tau$ be a non-associative monomial on a free algebra and consider $P_1\cup P_2$ a partition of the variables $x_1,\dots,x_n$.

Then the summand of $\Delta(m)$ associated to the partition $P_1\cup P_2$ is given by the following procedure:

\begin{enumerate}
\item In the tree $\tau$, label every leaf with $1$ or $-1$ depending on the partition of the variables. 
\item Next, go down the tree associating to every internal node a label $0,1$ or $-1$ if both branches come from a node with that label; but associate $0$ if the branches come from nodes with different labels.
\item Finally, going up from the root, and considering all the internal nodes do the following:
\begin{itemize}
\item For nodes with labels $1$ or $-1$ do nothing.
\item For nodes with label $0$ that came from two $0$ labeled branches do nothing.
\item For nodes with label $0$ that came from at least one non zero node, apply $\alpha$ to the leaves of the other branch with opposite label.
\end{itemize}
\end{enumerate}
\end{lem}

\begin{proof}
We only need to check the three cases of the lemma and show how $\alpha$ is applied.

If we have the case $(a\otimes 1)(b\otimes 1)$, it means that $a,b$ are in the same part of the partition and in this case $(ab\otimes 1)$ is the solution where there are no $\alpha$ applications to any variable. This is case $(1)$ in our list.

Also there is the possibility $(a\otimes b)(c\otimes d)$; where there is no $\alpha$ application either. This is case $(2)$ in our list

If we have the case $(a\otimes b)(c\otimes 1)$ it means that one side comes from a $0$ node and the other comes from either $1$ or $-1$. In this case we have $(ac\otimes \alpha(b)$ which is the application of $\alpha$ to every variable in the side opposite to $c$.
It may also happend that $(1\otimes b)(c\otimes 1)=\alpha(c)\otimes \alpha(b)$; which is the same case applied to both branches. This is case $(3)$ in our list.
\end{proof}

Consider now $(A,\mu,\alpha)$ a free Hom-associative unitary algebra, it should have a natural coalgebra structure to be able to consider its primitive elements.
In the case of the tensor algebra, which is the free associative algebra, the coproduct that appears naturally is called shuffle coproduct. It is the one that we consider in this context, but with the exception that our algebra is not associative, but Hom-associative.  Hence, since $u\cdot x=\alpha(x)$, the coproduct is no longer the shuffle coproduct.

The coproduct on the free Hom-associative algebra will be called the $\alpha$-shuffle coproduct; and in case $\alpha=id$, it is the shuffle coproduct on the tensor algebra.

\subsection{Universal enveloping algebras}

In this section, we aim to find an appropriate universal enveloping algebra to any Hom-Sabinin algebra.

\begin{prop}
Let $(A,\alpha)$ be a Hom-algebra. Then $YIII_{hom}(A,\alpha)$ defined in \ref{YIII} is a Hom-Sabinin algebra.
\end{prop}

\begin{proof}
In any algebra; the $q$ operations give a Sabinin algebra. By the Main Theorem, the $q^\alpha$ operations give rise to a Hom-Sabinin algebra.
\end{proof}

The previous proposition shows that we have a functor $YIII_{hom}:\mathbf{Hom}$-$\mathbf{Alg}\to \mathbf{Hom}$-$\mathbf{SabAlg}$, from the category of Hom-algebras with one binary multiplication to the category of Hom-Sabinin algebras. This functor has a left adjoint functor that we shall denote by $U_{hom}:\mathbf{Hom}$-$\mathbf{SabAlg}\to \mathbf{Hom}$-$\mathbf{Alg}$.

 The image $U_{hom}(S,\alpha)$ can be defined as the quotient of the free Hom-algebra $(\mathbb{K}\{S\},\hat{\alpha})$ generated by $S$; by the ideal $I$ generated by $\{<u;a,b>+q^\alpha(u,a,b)-q^\alpha(u,b,a)|\ u\in T(S);\ a,b\in S\}\cup\{[a,b]+<a,b>|\ a,b\in S\}\cup\{\Phi(u,v)-\sum\frac{1}{n!m!}q^\alpha(u,v)|\ u,v\in T(S)\}$. The endomorphism is $$\alpha_U(\overline{(s_1\dots s_n)_\tau})=(\overline{\alpha(s_1)}\dots \overline{\alpha(s_n)})_\tau$$ for $s_i\in S$ and $\tau$ a tree of $n$ leaves.

This functor $U_{hom}$ is well defined since $\hat{\alpha}(I)\subset I$ for $\hat{\alpha}$ the endomorphism of the free algebra $\mathbb{K}\{S\}$. We call this functor the universal enveloping Hom-algebra; in view of the following  result:

\begin{prop}
The functor $U_{hom}$ is left adjoint to $YII_{hom}$.
\end{prop}

\begin{proof}
We have to show that $$Hom_{\mathbf{Hom}\text{-}\mathbf{Alg}}(U_{hom}(S,\alpha),(A,\beta))\cong Hom_{\mathbf{Hom}\text{-}\mathbf{SabAlg}}((S,\alpha),YIII_{hom}(A,\beta)),$$ where the first maps are Hom-algebra morphisms and the second maps are Hom-Sabinin algebra morphisms.

Consider $f:(S,\alpha)\to YIII_{hom}(A,\beta)$ a Hom-Sabinin algebra morphism. Then it is a linear map $f:S\to A$ where $f\circ\alpha=\beta\circ f$ and defines $F:\mathbb{K}\{S,\alpha\}\to (A,\beta)$ a Hom-algebra morphism in a unique way where $F|_S=f$.
\begin{eqnarray*}
&& F(<u;a,b>+q^\beta(u,a,b)-q^\beta(u,b,a))\\ && =f(<u;a,b>) \\ && +q^\beta(f(u_1)\dots f(u_n),f(a),f(b))-q^\beta(f(u_1)\dots f(u_n),f(b),f(a))\\ && =<f(u_1)\dots f(u_n);f(a),f(b)>+q^\beta(f(u_1)\dots f(u_n),f(a),f(b))\\ && -q^\beta(f(u_1)\dots f(u_n),f(b),f(a))=0.
\end{eqnarray*}
$$F([a,b]+<a,b>)=[fa,fb]+f<a,b>=[fa,fb]+<fa,fb>=0.$$
\begin{eqnarray*}
&& F(\Phi(u,v)-\sum\frac{1}{n!m!}q^\beta(u,v))\\ && =f(\Phi(u,v))-\sum\frac{1}{n!m!}q^\beta(f(u_1)\dots f(u_n),f(v_1)\dots f(v_{m-1}),f(u_m)))\\ && =\Phi(f(u_1)\dots f(u_n),f(v_1)\dots f(v_{m-1}),f(v_m))\\ && -\sum\frac{1}{n!m!}q^\beta(f(u_1)\dots f(u_n),f(v_1)\dots f(v_{m-1}),f(u_m))).
\end{eqnarray*}

Therefore, $I\subseteq ker(F)$ and $\widehat{F}:U_{hom}(S)\to A$ as $\widehat{F}(x+I)=f(x)$ is a Hom-algebra morphism. This is a functor from $ Hom((S,\alpha),YIII_{hom}(A,\beta))$ to $ Hom(U_{hom}(S,\alpha),(A,\beta))$.

Consider $g:U_{hom}(S)\to A$, then $$YIII_{hom}(g):YIII_{hom}(U_{hom}(S,\alpha))\to YIII_{hom}(A,\beta)$$ is a Hom-Sabinin algebra morphisms; and also is $\pi:S\to YIII_{hom}(U_{hom}(S,\alpha))$ defined as $\pi(s)=s+I$.\\ In conclusion, $YIII_{hom}(g)\circ\pi:(S,\alpha)\to  YIII_{hom}(A,\beta)$ is a Hom-Sabinin algebra morphism and this is a functor from $Hom(U_{hom}(S,\alpha),(A,\beta))$ to $ Hom((S,\alpha),YIII_{hom}(A,\beta))$.

This two functors are inverses of each other, proving our claim.
\end{proof}

The question of whether $\pi$ is injective or not is still open. We conjecture that in general $\pi$ might not be injective depending on the properties of $\alpha$. The proof in case of Sabinin algebras ($\alpha=id$) can be read in \cite{UnivSab}.

\subsection{Universal envelopes as Hom-bialgebras}

We consider in this section  Hom-Hopf algebra, which are $(\alpha,id)$-Hom-Hopf algebra with respect to the definition given in \cite{MakhSilvHOMBialg}, where the structure maps for the product and the coproduct may be different. We show that  the universal enveloping algebra of  a Hom-Lie algebra is endowed with such a structure while the universal enveloping algebra of a Hom-Sabinin algebra is endowed with a Hom-bialgebra structure.

\begin{thm}
For any Hom-Sabinin algebra $(S,\alpha)$;  its universal enveloping algebra $U_{hom}(S,\alpha)$ is a Hom-bialgebra.
\end{thm}

\begin{proof}
We only need to check that the coproduct of the free algebra generated by $S$ can be taken to the quotient. This is true because the ideal is generated by primitive elements and $\Delta$ is an algebra morphism.

Since $\Delta$ is a morphism, $\Delta(I)$ is an ideal on $\mathbb{K}\{S\}\otimes \mathbb{K}\{S\}$. Also for any generator $p\in I$, $\Delta(p)=1\otimes p+p\otimes 1\in B\otimes I+I\otimes B$. Hence $\Delta(I)\subseteq B\otimes I+I\otimes B$; by ideals  properties.
\end{proof}

\begin{defn}
A Hom-Hopf algebra is a unitary and counitary Hom-bialgebra $(H,\mu,u,\Delta, \epsilon,\alpha)$ which is Hom-associative, coassociative and such that there is a map $S:H\to H$ satisfying 
$$\mu\circ(id\otimes S)\circ\Delta=u\circ\epsilon, \; 
\mu\circ(S\otimes id)\circ\Delta=u\circ\epsilon\;\text{and }S\circ\alpha=\alpha\circ S.$$ The map  $S$ is called an antipode.
\\ This definition corresponds to that of $(\alpha,id)$-Hom-Hopf algebra given in \cite{MakhSilvHOMBialg}.
\end{defn}

\begin{prop}
Let $(H,\mu,u,\Delta, \epsilon,\alpha)$ be a  Hom-associative, coassociative Hom-bialgebra. If  $H$ admits an antipode, then it is unique up to $im(\alpha)$; i.e., $\alpha\circ S=\alpha\circ S^\prime$ if $S,S^\prime$ are two antipodes for $H$.
\end{prop}

\begin{proof}
\begin{eqnarray*}
&& \alpha\circ S=\mu\circ (\alpha\circ S\otimes[\mu\circ(id\otimes S^\prime)\circ\Delta])\circ \Delta\\ &&=
\mu\circ (\alpha\otimes\mu)\circ (S\otimes id\otimes S^\prime)\circ(id\otimes \Delta)\circ \Delta
\\ &&=
\mu\circ (\mu\otimes\alpha)\circ (S\otimes id\otimes S^\prime)\circ(\Delta\otimes id)\circ \Delta\\  &&=
\mu\circ ([\mu\circ(S\otimes id)\circ\Delta]\otimes\alpha\circ S^\prime)\circ \Delta=\alpha\circ S^\prime
.\end{eqnarray*}
\end{proof}

\begin{defn}
Let $X$ be a finite set and  $X_i$ denote a copie  of that set for $i\in\mathbb{N}$. Define $\alpha:\bigcup X_i\to \bigcup X_i$ as $\alpha(x_i)=x_{i+1}$ and consider $\mathcal{F}_\alpha(X)$ the free algebra generated by $\bigcup X_i$; and extend $\alpha$ as an endomorphism.

Finally, make a  quotient of that free algebra by the ideal generated by $I=\langle\alpha(a)(bc)-(ab)\alpha(c)|\ a,b,c\in \mathcal{F}(X)\rangle$.  We call that quotient the free Hom-associative algebra generated by $X$ and we denote it by $\mathcal{F}_{\alpha,ass}(X)$. 

Note that this is a free construction from the category $\mathbf{Set}$, or $\mathbf{FinVec}$, the category of finite dimensional vector spaces; and not from the category of Hom-$\mathbb{K}$-modules. The difference lies in the fact that in this case, $\alpha$ is freely defined  and in the latter case, $\alpha$ is already given and is only extended.
\end{defn}

\begin{prop}
The map $\alpha$ is injective on  $\mathcal{F}_{\alpha,ass}(X)$.
\end{prop}

\begin{proof}
First note that $\alpha:\mathcal{F}_{\alpha,ass}(X)\to\mathcal{F}_{\alpha,ass}(X)$ is well defined as $\alpha(I)\subseteq I$.

Consider $\alpha([s])=0$ for some $[s]\in \mathcal{F}_{\alpha,ass}(X)$; then $\alpha(s)\in I$ considered as $\alpha:\mathcal{F}_\alpha(X)\to \mathcal{F}_\alpha(X)$. Hence $\alpha(s)(bc)\in I$ and also $(sb)\alpha(c)\in I$ for all $b,c\in \mathcal{F}_\alpha(X)$. By the definition of $I$ and the fact that the algebra is free, it is only possible that $sb\in I$. By the same argument as before, we have   $s\in I$. 

Therefore,  $[s]=0$ and $\alpha$ is injective.
\end{proof}

\begin{prop}
$ \mathcal{F}_{\alpha,ass}(X)$ is a Hom-Hopf algebra.
\end{prop}

\begin{proof}
Define $\Delta(x_i)=1\otimes x_i+x_i\otimes 1$ for $x_i\in X_i$ and extend it as an algebra morphism. Also consider $\epsilon(x_i)=0$ and $\epsilon(1)=1$ and extend them as  algebra morphisms. Set  also $S(x_i)=-x_i$, $S(1)=1$ and $S(u\ v)=S(v)S(u)$ and extend it  linearly. We need to show that $\sum u_{(1)}S(u_{(2)})=\epsilon(u)$.

The proof is trivial for the case of $u=1,x_i$. Assume now that $|u|> 1$.

Then $u=a\ b$ where $|a|,|b|\geq 1$, and 
\begin{eqnarray*}
&& \alpha(\sum u_{(1)}S(u_{(2)}))=\alpha(\sum (a_{(1)}b_{(1)})(S(b_{(2)})S(a_{(2)})))\\&& =\sum \alpha(a_{(1)}b_{(1)})(\alpha(S(b_{(2)}))\alpha(S(a_{(2)})))=\sum ((a_{(1)}b_{(1)})\alpha(S(b_{(2)})))\alpha^2(S(a_{(2)}))\\&& =\sum (\alpha(a_{(1)})(b_{(1)}S(b_{(2)})))\alpha^2(S(a_{(2)}))=\sum \alpha(a_{(1)})\alpha^2(S(a_{(2)}))\epsilon(b)=0.
\end{eqnarray*}

By injectivity, it follows that $\sum u_{(1)}S(u_{(2)})=\epsilon(u)$.
\end{proof}

\begin{cor}
Let $(L,\beta)$ be a Hom-Lie algebra. Then its universal enveloping algebra $U_{hom}(L,\beta)$ as defined in \cite{DYEnvHomLie} is a Hom-Hopf algebra.
\end{cor}

\begin{proof}
Since $\mathcal{F}_{\alpha,ass}(L)$ is a Hom-Hopf algebra and $U_{hom}(L,\beta)$ is defined as a quotient and generated by primitive elements. This quotient preserves the coproduct. In particular,  the ideal is generated by $\{[x,y]-xy+yx|\ x,y\in L\}\cup\{x_i-\beta^{i-1}(x_1)|\ x_i\in X_i,\ i>1\}$.
\end{proof}

Here are some examples of the equation $\sum u_{(1)}S(u_{(2)})=\epsilon(u)$, for small  $|u|$ cases.

\begin{ex}\ 
\begin{itemize}
\item If $u=x$ then $\sum u_{(1)}S(u_{(2)})=x\ 1+1 S(x)=\alpha(x)-\alpha(x)=0.$
\item If $u=(xy)$ then
 \begin{eqnarray*}
 && \sum u_{(1)}S(u_{(2)})=xy\ 1+1 S(xy)+\alpha(x)S(\alpha(y))+\alpha(y)S(\alpha(x))\\&& =\alpha(xy)+\alpha(yx)-\alpha(xy)-\alpha(yx)=0.
 \end{eqnarray*}
\item If $u=(xy)z$ then  
\begin{align*}
& \sum u_{(1)}S(u_{(2)})=(xy)z\ 1+1 S((xy)z)+\alpha^2(x)S(\alpha(y)z)+\alpha^2(y)S(\alpha(x)z)+ \\ &\alpha(xy)S(\alpha (z))+
\alpha(z)S(\alpha(xy))+ \alpha(y)zS(\alpha^2(x))+\alpha(x)zS(\alpha^2(y))\\ &   =
\alpha((xy)z)-\alpha(z(yx))+\alpha^2(x)(z\alpha (y))+\alpha^2(y)(z\alpha (x))-\alpha((xy)z)+\\ & 
\alpha(z)(\alpha(yx))-(\alpha(y)z)\alpha^2(x)-(\alpha(x)z)\alpha^2(y) \\ & =
 \alpha^2(x)(z\alpha (y))+\alpha^2(y)(z\alpha (x))-(\alpha(y)z)\alpha^2(x)-(\alpha(x)z)\alpha^2(y)=0.
\end{align*}
\item If $u=(a(bc))d$ then
\begin{align*}
& \sum u_{(1)}S(u_{(2)})=(a(bc))d\ 1+1\ S((a(bc))d)+\alpha(a(bc))S(\alpha(d))\\&  +\alpha(d)S(\alpha(a(bc)))+  
 \alpha^2(bc)S(\alpha(a)d)+(\alpha(a)d)S(\alpha^2(bc))\\&  +\alpha^3(b)S((a\alpha(c))d)+
  ((a\alpha(c))d)S(\alpha^3(b))+
\alpha^3(c)S((a\alpha(b))d)\\& +((a\alpha(b))d)S(\alpha^3(c))+\alpha^2(a)S(\alpha(bc)d)+(\alpha(bc)d)S(\alpha^2(a))\\& +
\alpha(a\alpha(b))S(\alpha^2(c)d)+(\alpha^2(c)d)S(\alpha(a\alpha(b)))\\& +
\alpha(a\alpha(c))S(\alpha^2(b)d)+(\alpha^2(b)d)S(\alpha(a\alpha(c)))\\&
  =
 \alpha((a(bc))d)+\alpha(d((cb)a))-\alpha((a(bc))d)-\alpha(d((cb)a))+\\& 
  \alpha^2(bc)(d\alpha(a))+(\alpha(a)d)\alpha^2(cb)-\alpha^3(b)(d(\alpha(c)a))\\& -((a\alpha(c))d)\alpha^3(b)-
\alpha^3(c)(d(\alpha(b)a))-((a\alpha(b))d)\alpha^3(c)\\& -\alpha^2(a)(d\alpha(cb)) -(\alpha(bc)d)\alpha^2(a)+
\alpha(a\alpha(b))(d\alpha^2(c))\\& +(\alpha^2(c)d)(\alpha(\alpha(b)a))+ 
\alpha(a\alpha(c))(d\alpha^2(b))+(\alpha^2(b)d)(\alpha(\alpha(c)a))=0.
\end{align*}
\end{itemize}
\end{ex}

Note that this last universal enveloping algebra may be different from our definition since it is adjoint to $YIII_{hom}:\mathbf{Hom}$-$\mathbf{Ass}\to \mathbf{Hom}$-$\mathbf{Lie}$, which considers only the Hom-associative algebras and not all Hom-algebras where $YIII_{hom}(A,\alpha)$ is a Hom-Lie; called Hom-Lie admissible algebra.

\subsection{An example: Antisymmetric algebras}

As an example, we will show the relation between algebras with an antisymmetric bilinear product, and connected $(0,id)$-Hom-Hopf algebras.

\begin{prop}
There is an equivalence of categories between Hom-Lie algebras where $\alpha=0$ and algebras with an antisymmetric bilinear product.
\end{prop}

\begin{proof}
Let $(L,[\ ,\ ],0)$ be a Hom-Lie algebra. Then $[\ ,\ ]$ is an antisymmetric bilinear product.

Let $(L,\cdot)$ be an algebra with an antisymmetric bilinear product. Then the Hom-Jacobiator $J_\alpha(x,y,z)=0$ for this product in case $\alpha=0$. Hence $(L,\cdot,0)$ is a Hom-Lie algebra where $\alpha=0$.

Let  $f:L_1\to L_2$  be a morphism of Hom-Lie algebras.  Then $f$ is a morphism of the corresponding antisymmetric algebras.

Consider finally, $f:L_1\to L_2$ a morphism of antisymmetric algebras. Then $f\circ 0=0=0\circ f$. Therefore, $f$ is a morphism of the corresponding Hom-Lie algebras where $\alpha_1=\alpha_2=0$.
\end{proof}

\begin{prop}
There is an equivalence of categories between Hom-associative algebras with $\alpha=0$ and magmatic algebras; i.e., algebras with one binary operation.
\end{prop}

\begin{proof}
The Hom-associativity is trivially satisfied for any algebra in case $\alpha=0$.
\end{proof}

\begin{cor}
The category of $(0,id)$-$Hom$-$Hopf$-algebras is equivalent to the category of algebras which are magmatic algebras; i.e., algebras with just one product, and that product is bilinear; with an element $u$ such that $uH=Hu=0$.

Since $\mathbb{K}\langle u\rangle$ is an ideal one can also consider only the algebra $H/\mathbb{K}\langle u\rangle$ and the element $u$ can be recovered to define the structure of a $(0,id)$-$Hom$-$Hopf$-algebra.
\end{cor}

\begin{proof}
Let $H$ be a $(0,id)$-Hom-Hopf algebra. Then it is a magmatic algebra with an element $u$ as specified; by forgetting the coalgebra structure. Let's denote this construction as the $F$ functor.

Consider any algebra $(H,\cdot)$ with an element $u\in H$ such that $uH=Hu=0$. Define a complement $P\oplus \mathbb{K}\langle u\rangle=H$ and a coproduct as $\Delta(u)=u\otimes u$ and $\Delta(x)=x\otimes u+u\otimes x$ for $x\in P$.
Also a counit as $\epsilon(u)=1$ and $\epsilon(x)=0$ for any $x\in P$. Then $(H,\cdot,u,\Delta,\epsilon)$ is a $(0,id)$-Hom-Hopf algebra. Let's denote this construction as the $G$ functor

This two constructions are inverse to one another. It is obvious that $F\circ G=ID$ since the product of the algebra is not modified by $G$; and the coalgebra construction of $G$ is forgotten by $F$.

Let $H$ be a $(0,id)$-Hom-Hopf algebra. First, let's prove that $Prim\ H$ is a subalgebra of $H$:
$$\Delta(xy)=\Delta(x)\Delta(y)=(x\otimes u+u\otimes x)(y\otimes u+u\otimes y)=xy\otimes u+u\otimes xy.$$

Since $(Prim\ H)^*\oplus H_0^*$ generates $Gr(H^*)$, see for instance \cite[Lemma 5.6.6]{SMon93}, it follows that $H=H_0\oplus Prim\ H$. The coproduct is then fixed as $\Delta(u)=u\otimes u$ and $\Delta(x)=x\otimes u+u\otimes x$ for $x\in Prim\ H$. Therefore $G\circ F=ID$.
\end{proof}

\begin{prop}
The Hom-associative universal enveloping algebra of a Hom-Lie algebra $L$ where $\alpha=0$ is given by $U_\alpha(L)=\mathbb{K}\{L\}/\langle xy-yx-[x,y]|\ x,y\in L\rangle$.
\end{prop}

\begin{proof}
Since $\alpha=0$ by the previous proposition, any algebra is Hom-associative in this case, and the result follows.
\end{proof}

\begin{cor}
We have $\pi(L)\cong L$ as vector spaces, where $\pi:\mathbb{K}\{L\}\to U_\alpha(L)$ is the canonical quotient map.
\end{cor}
\begin{proof}
The enveloping algebra $U(L)$ has a filtration given by $F_n=\pi(\bigoplus_{k=1}^n L^{\otimes k})$ and the associated graded algebra is the free commutative algebra generated by $L$. We have  he following result.
\end{proof}

\begin{thm}The map 
$\pi:(L,[\ ,\ ],0)\to YIII_\alpha(U_\alpha(L))$ is an injective morphism of Hom-Lie algebras; even more, the Hom-Lie algebra $L$ can be recovered from the $(0,id)$-Hom-Hopf algebra $U(L)$ as $F_1/F_0$, where $F_n$ is the associated filtration defined previously.
\end{thm}

\begin{thm}
Let $H$ be a $(0,id)$-Hom-Hopf algebra with a filtration where $Gr(H)$, the graded algebra associated to this filtration, is free commutative generated by $H_1$. Then $H_1<YIII_\alpha(H)$ and $H\cong U_\alpha(H_1/H_0,[\ ,\ ])$.
\end{thm}

\begin{proof}
Since $Gr(H)$ is commutative, it follows that $[H_1,H_1]\subseteq H_1$. Hence $H_1$ is closed by the commutator product. Also $H_0\cdot H=0$ and hence $H_1/H_0$ is a quotient of $(H_1, [\ ,\ ])$.

Also there is a morphism $f:H_1/H_0\to YIII_0(H/H_0)$ given by the inclusion of $H_1$ into $H$. By the universal property, then there is $F:U_\alpha(H_1/H_0)\to H/H_0$ a $(0,id)$-Hom-Hopf algebra morphism. This map is given by $F([\overline{h_1}\dots \overline{h_n}])=\overline{h_1\dots h_n}$.

The map $F$ also respects the filtrations; since $F(F_n)\subseteq H_n/H_0$. Then there is a map $Gr F: \mathbb{K}\{H_1/H_0\}/\langle xy-yx|\ x,y\in H_1/H_0\rangle \to Gr(H/H_0)\oplus H_0$ that is an isomorphism, hence $F$ is an isomorphism when identifying $F(F_0)= H_0$.
\end{proof}

We have found then an equivalence between this two categories, but the class of $(0,id)$-Hom-Hopf algebras $H$ with a filtration where $Gr(H)$, the graded algebra associated to this filtration, is free commutative generated by $H_1$; should be studied in more detail. This problem remains open whether our hypothesis on the filtration are necessary or less conditions already imply them. 

Remember that  every connected Hopf algebra has a coradical filtration such that $Gr(H)$ is free in the class of commutative and associative algebras; and is generated by $H_1$.

\bibliographystyle{plainurl}
\bibliography{ArticlesBib}

\end{document}